\documentclass[1p]{elsarticle}
\usepackage{amssymb}
\usepackage{amsfonts}
\usepackage{amsmath}

\usepackage[polish, english]{babel}
\usepackage{polski}
\usepackage{tikz}

\newtheorem{theorem}{Theorem}

\newtheorem{corollary}[theorem]{Corollary}
\newtheorem{lemma}[theorem]{Lemma}
\newtheorem{observation}[theorem]{Observation}

\newproof{pf}{Proof}

\begin{document}

\title{Asymptotically Tight Bound for the Conflict-Free Chromatic Index} 
\tnotetext[t1]{This research was supported by the AGH University of Krakow under grant no. 16.16.420.054, funded by the Polish Ministry of Science and Higher Education.}

\author[agh]{Mateusz Kamyczura}
\ead{kamyczuram@gmail.com}

\author[agh]{Jakub Przyby{\l}o}
\ead{jakubprz@agh.edu.pl}

\address[agh]{AGH University of Krakow, al. A. Mickiewicza 30, 30-059 Krakow, Poland}

\begin{abstract}
The conflict-free chromatic index of a graph $G$ is the minimum number of colours in an edge colouring of $G$ such that the neighbourhood of every edge contains a colour appearing exactly once. Its vertex analogue is the conflict-free chromatic number. These two parameters naturally coincide when the second is applied to the line graph of $G$. It is known that two variants of the latter parameter exhibit substantially different behaviour.

For closed vertex neighbourhoods, where each vertex belongs to its own neighbourhood, it is known that \(O(\ln^2 \Delta)\) colours suffice, where \(\Delta\) denotes the maximum degree of \(G\), and this bound is tight in order. In contrast, for open neighbourhoods, the corresponding parameter can be as large as \(\Delta+1\), but is bounded above by \(O(\ln^{2+\varepsilon} \Delta)\) for claw-free graphs. Since line graphs are claw-free, this yields the best known general upper bound for the edge analogue in the setting of open neighbourhoods.

For closed edge neighbourhoods, a stronger general upper bound of \(3\log_2 \Delta + 4\) is known. In this paper, we show that for both variants, the conflict-free chromatic index is bounded above by \((1+o(1))\log_2 \Delta\).  Since complete graphs require at least $(1 - o(1)) \log_2 \Delta$ colours in the closed as well as the open setting, 
our result is asymptotically tight in order and in the leading constant. Moreover, we strengthen this conclusion by showing that $(1 - o(1)) \log_2 \Delta$ colours are also typically necessary, as we prove this asymptotically almost surely for random graphs in both dense and relatively sparse regimes. 

Our proofs combine the probabilistic method with deterministic graph decomposition techniques, as well as new results relating the parameters under consideration with the chromatic number of a graph.
\end{abstract}


\maketitle

\section{Introduction}

All graphs considered throughout the paper are assumed to be simple and finite. 
An edge colouring of a graph is called \emph{conflict-free} if the neighbourhood of every edge includes an element whose colour is \emph{unique}, that is different from all remaining colours appearing in the neighbourhood, unless the neighbourhood is empty. Vertex \emph{conflict-free} colourings of graphs are defined analogously, with the colour uniqueness requirement imposed on vertex neighbourhoods.
In either setting, both open and closed neighbourhoods are considered, leading to graph invariants exhibiting significantly distinct behaviours. Let us specify that the \emph{closed neighbourhood} of an edge $e=uv$ is the set of all edges incident with $u$ or $v$, while $e$ itself is not included in its \emph{open neighbourhood}.

The common root of these concepts are conflict-free colourings of hypergraphs, originally arising from problems in wireless communication and frequency assignment. In such settings, available frequencies are modeled as colours, and successful communication requires that within a given region (or neighbourhood) there exists a uniquely occurring frequency, ensuring the absence of interference. This naturally leads to conflict-free colouring models, which in particular specialize to graph-theoretic variants. See~\cite{EvenEtAl,KostochkaEtAl,SmorodinskyPhd,SmorodinskyApplications} for results related with the general hypergraph setting. 

Observe that for closed neighbourhoods, proper colourings are trivially conflict-free, both for vertex and edge colourings. Consequently, $\Delta+1$ colours always suffice, where $\Delta$ denotes the maximum degree of a graph. Nevertheless, significantly stronger bounds usually hold, as conflict-free colourings need not be proper.
Let us denote by $\chi_{\rm ccf}(G)$ the \emph{closed conflict-free chromatic number} of a graph $G$, that is the least number of colours in a conflict-free vertex colouring of $G$ within the closed neighbourhoods regime.
The correct order of magnitude for an upper bound on this parameter was established in the fundamental papers by Glebov, Szab\'o, and Tardos~\cite{GlebovEtAl}, and by Bhyravarapu, Kalyanasundaram, and Mathew~\cite{Hindusi}. They showed that 
\begin{equation}\label{Eq1.0}
\chi_{\rm ccf}(G)=O(\ln^2 \Delta)
\end{equation}
for all graphs, while there exists a family of graphs requiring 
$\Omega(\ln^2 \Delta)$  
colours.
The existence of the family exploits random graphs, end its resulting members have highly diversified degrees, see~\cite{GlebovEtAl}. 
On the other hand, it is known that $\chi_{\rm ccf}(G)=O(\ln\Delta)$ for regular and nearly regular graphs $G$, see~\cite{DebskiPrzybylo}.

Focussing on open vertex neighbourhoods, rather than closed ones, gives rise to the \emph{open conflict-free chromatic number}, $\chi_{\rm ocf}(G)$, which exhibits  a completely different behaviour. Its value can in particular be as large as $\Delta+1$, as witnessed by subdivisions of complete graphs. However, such graphs have a specific structure and contain vertices of very small degree, which leaves room for improvement under additional assumptions. Pach and Tardos~\cite{PachTardos} proved in particular that
\begin{equation}\label{Eq1.1}
\chi_{\rm ocf}(G)=O(\ln^2\Delta)
\end{equation}
provided that $G$ has minimum degree $\delta=\Omega(\ln\Delta)$.
On the other hand,  Bhyravarapu, Kalyanasundaram and Mathew~\cite{Hindusi3,Hindusi2}
showed that 
\begin{equation}\label{Eq1.2}
\chi_{\rm ocf}(G)=O(\ln^{2+\varepsilon}\Delta)
\end{equation}
 for $K_{1,r}$--free graphs $G$, for any fixed $r\geq 3$ and $\varepsilon>0$.
See~\cite{Abel2018} for other results on conflict-free vertex colourings.

Conflict-free \emph{edge} colourings have been studied more extensively for closed neighbourhods. 
In~\cite{DebskiPrzybylo} it was observed that the complete graphs witness that the corresponding parameter, the \emph{closed conflict-free chromatic index}, $\chi'_{\rm ccf}(G)$, can be as large as $(1-o(1))\log_2\Delta$. 
On the other hand, since conflict-free edge colourings of a graph $G$ correspond one-to-one to conflict-free vertex colourings of the line graph of $G$, it was originally known by~(\ref{Eq1.0}) that $\chi'_{\rm ccf}(G)\leq O(\ln^2\Delta)$.
However, the results of~\cite{DebskiPrzybylo}  demonstrated that conflict-free edge colourings allow for a substantial reduction in the number of required colours: in particular, it was shown that
$\chi'_{\rm ccf}(G) = O(\log_2\Delta)$, and more precisely that
$\chi'_{\rm ccf}(G)\leq C\log_2\Delta$ for every graph $G$ with  $\Delta=\Delta(G)>1$,
where $C>337$ is an absolute constant.
This bound was subsequently improved to $\chi'_{\rm ccf}(G)\leq 3\log_2\Delta+4$ for all graphs $G$
in~\cite{MPK}, and to  $\chi'_{\rm ccf}(G)\leq (1+o(1))\log_2\Delta$ for nearly regular graphs in~\cite{KPreg}; see also~\cite{GuoTrees24} for other results. 
The main result of this paper implies that 
\begin{equation}\label{CCFbound}
\chi'_{\rm ccf}(G)\leq (1+o(1))\log_2\Delta
\end{equation}
for all graphs  with $\Delta>1$, thus providing a general upper bound that is asymptotically tight, both in order and in the leading multiplicative constant.

For the \emph{open conflict-free chromatic index}, $\chi'_{\rm ocf}(G)$, until recently it was only known that, in general, $\chi'_{\rm ocf}(G)\leq O(\ln^{2+\varepsilon}\Delta)$, which follows by~(\ref{Eq1.2}), as line graphs of all graphs are claw-free, that is $K_{1,3}$-free. In~\cite{KPreg2} it was shown only that $\chi'_{\rm ocf}(G)\leq (1+o(1))\log_2\Delta$ for nearly regular graphs, strengthening the direct consequence of~(\ref{Eq1.1}).
At the same time, it was proven that complete graphs demonstrate that no better general upper bound than 
 $(1-o(1))\log_2\Delta$ can be guaranteed for $\chi'_{\rm ocf}(G)$.
 Due to several adaptations, our approach turns out to be sufficiently robust to establish an asymptotically optimal upper bound also for the open conflict-free chromatic index. More precisely, we prove that
\begin{equation}\label{OCFbound}
\chi'_{\rm ocf}(G)\leq (1+o(1))\log_2\Delta
\end{equation}
for all graphs with $\Delta>1$, which is tight in the same sense as for $\chi'_{\rm ccf}(G)$.

We shall furthermore demonstrate that our main results in~(\ref{CCFbound}) and~(\ref{OCFbound})  are tight in a significantly stronger sense, as they effectively determine the values of the considered parameters for `almost all' graphs, up to lower-order terms. To this end, we prove that asymptotically almost surely the closed and open conflict-free chromatic indices are at least $(1 - o(1)) \log_2 \Delta$ for random graphs $\mathbb{G}_{n,p}$ across a broad range of $p$, including $p = n^{-\beta}$ for any fixed $0 \leq \beta < 1$. This covers both dense and relatively sparse regimes; see Section~\ref{SectionRandomComments} for additional discussion and more precise formulations of these results.

The  conflict-free constraint in open neighbourhoods has recently been intensively studied  also for \emph{proper} vertex colourings (note that the closed neighbourhood setting would not impose any additional condition on proper colourings), see e.g.~\cite{TreesCF,OuterplanarCF,CPS,EUN,DWC-CHL,Fab,KP-proper}. One of the most interesting direction in this area concerns a generalisation of the problem in which multiple unique colours, rather than just one, are required in every neighbourhood, as proposed by Cho, Choi, Kwon and Park~\cite{EUN2}.
The extreme case, in which all colours in each neighbourhood are required to be unique, is equivalent to the well-established and well-studied notion of distance-two colourings, which may require a number of colours quadratic in  $\Delta$.
However, many intermediate variants of this parameter, where not all colours in the neighbourhoods are required to be unique, typically admit colourings using no more than $O(\Delta)$ colours. 
This leads naturally to the question of the threshold between these two types of behaviour of the corresponding invariants. Intriguing and involved results related with this direction have recently been obtained in particular by 
Liu and Reed~\cite{Liu-Reed,Liu-Reed_2}, as well as by Chuet, Dai, Ouyang and Pirot~\cite{CF-Pirot}.
See also~\cite{CFChoosabilityHardness,CHL} for list variants of these problems.

In the next section, we formalise the notation and state our main result, along with several tools from the probabilistic method and graph decomposition theory. In Section~\ref{DeterministicSection}, we present several observations concerning the studied parameters for certain families of graphs with simple structure, and relate them with the chromatic number of a graph, $\chi(G)$. While such a relation was already established for $\chi'_{\rm ccf}$ in~\cite{DebskiPrzybylo}, it remained unknown and appeared more elusive in the case of $\chi'_{\rm ocf}$. All these results shall be used in the following section, where we prove our main result, Theorem~\ref{GoodAsymptCFE-thm}, implying~(\ref{CCFbound}) and~(\ref{OCFbound}). The final section is devoted to random graphs and lower bounds for the studied invariants, and includes several additional remarks and further perspectives.

\section{Preliminaries and main result}

Let \( G = (V, E) \) be a graph. For any \( v \in V \) and subsets \( U \subseteq V \), \( F \subseteq E \), we use the following notation. 
By \( N_G(v) \) we denote the set of neighbours of \( v \) in \( G \). 
By \( d_G(v) \), abbreviated to \( d(v) \) when \( G \) is unambiguous, we denote the degree of \( v \) in \( G \). 
Further, \( d_U(v) \) denotes the number of neighbours of \( v \) in \( U \), while \( d_F(v) \) denotes the number of edges in \( F \) incident with \( v \).
We write \( G[U] = (U,E[U])\) for the subgraph of \( G \) induced by \( U \), i.e. with $E[U]$ including all edges of $E$ with both ends in $U$. 
Given an additional set \( U' \subseteq V \) with \( U' \cap U = \emptyset \), we denote by \( G[U, U'] \) the bipartite graph with vertex set \( U \cup U' \) and edge set \( E[U, U'] \) consisting of all edges in $E$ joining \( U \) with \( U' \).
For any positive integer \( k \), let \( [k] := \{1,2,\ldots,k\} \). 
An edge decomposition of \( G \) is a collection of its subgraphs \( G_i = (V_i, E_i) \), \( i \in [k] \), whose edge sets partition \( E \). 

For a vertex \( v \in V \), we denote by \( E(v) \) the set of edges incident with \( v \).
For an edge \( e = uv \in E \), we define its closed neighbourhood as \( E[e] = E(u) \cup E(v) \) and its open neighbourhood as \( E(e) = (E(u) \cup E(v)) \setminus \{e\} \). 
An assignment \( c : E \to C \) is called \emph{closed conflict-free} if, for every \( e \in E \), the set \( c(E[e]) \) contains a colour assigned to exactly one element of \( E[e] \). 
Similarly, \( c \) is called \emph{open conflict-free} if, for every \( e \in E \) with \( E(e) \neq \emptyset \), the set \( c(E(e)) \) contains a colour assigned to exactly one element of \( E(e) \).

Let us first observe that the two colouring notions are independent and not directly comparable, nor are the corresponding graph invariants. Indeed, one easily verifies that $\chi'_{\rm ccf}(C_5)=2<3=\chi'_{\rm ocf}(C_5)$, while $\chi'_{\rm ccf}(C_5^+)=3>2=\chi'_{\rm ocf}(C_5^+)$, where $C_5^+$ denotes the cycle $C_5$ with a single chord added. (Note also that these examples show $\chi'_{\rm ocf}$ may increase on subgraphs; to see the same for $\chi'_{\rm ccf}$, consider e.g. the graph obtained from $C_5^+$ by appending a pendant edge to a vertex of degree $2$ that is adjacent to exactly one vertex of degree $3$.)

Therefore, as we aim to give upper bounds for both parameters, to avoid repetitions, we make our task slightly more demanding and restrict attention to colourings that are conflict-free in both senses. Consequently, we investigate an auxiliary hybrid parameter that simultaneously bounds the two original ones. It shall also be convenient to extend our notation to partial edge colourings of $G$ (i.e.\ colourings that need not assign colours to all edges). Specifically, given a (partial) edge colouring $c$ of a graph $G=(V,E)$, we call a colour $\alpha$ \emph{unique} for an edge $e$ if $\alpha$ appears on some edge $e'\in E(e)$ in the open neighbourhood of $e$ and does not appear on any other edge $e''\in E[e]\setminus{e'}$ in the closed neighbourhood of $e$. If such a colour exists we say that $e$ is \emph{satisfied} by $c$ (or by $e'$ or by $\alpha$). From now on an edge colouring $c$ of $G$ shall be called \emph{conflict-free} if every non-isolated edge of $G$ is satisfied by $c$ -- note that such a colouring is both a closed conflict-free and an open conflict-free edge colouring of $G$. The minimum number of colours in a conflict-free colouring of all edges of $G$ shall be called the \emph{conflict-free chromatic index} of $G$, and denoted $\chi'_{\rm cf}(G)$. The key conclusions of our paper can be derived from the following result.

\begin{theorem}\label{GoodAsymptCFE-thm}
For every $\varepsilon>0$ there exists $\Delta_{\varepsilon}$ such that for every graph $G$ with maximum degree $\Delta\ge\Delta_{\varepsilon}$,
$$
\chi'_{\rm cf}(G)\leq (1+\varepsilon)\log_2\Delta.
$$
\end{theorem}

Since $\chi'_{\rm ccf}(G)$ and $\chi'_{\rm ocf}(G)$ are clearly bounded above by $\chi'_{\rm cf}(G)$ for every graph $G$, the corresponding upper bounds~(\ref{CCFbound}) and~(\ref{OCFbound}) follow; see Corollary~\ref{GoodAsymptCFE-cor} in Section~\ref{SectionRandomComments}. 

In our arguments we shall employ a collection of standard tools from the probabilistic method. We begin with the symmetric form of the Lov\'asz Local Lemma; see, for instance,~\cite{LLL}.

\begin{lemma}[Lov\'asz Local Lemma]\label{LLL}
Let $\Omega$ be a finite set of events in a probability space. Suppose that each event $A \in \Omega$ is mutually independent of all other events in $\Omega$ except for at most $D$ of them, and that $\Pr(A) \le p$ for every $A \in \Omega$. If
\[
e p (D + 1) \le 1,
\]
then
\[
\mathbb{P}\!\left( \bigcap_{A \in \Omega} A \right) > 0.
\]
\end{lemma}

We shall also repeatedly apply concentration inequalities of Chernoff type. The following two forms shall suffice for our purposes; see, e.g.,~\cite{RandOm,ChernoffBook}.

\begin{lemma}[Chernoff Bounds]\label{Chu+d}
Let $X_1,X_2,\ldots,X_n$ be a collection of independent Ber\-nou\-lli random variables where $\Pr(X_i = 1) = p_i$ for $i\in [n]$. Denote $X = \sum_{i=1}^n X_i$.
Then, for every positive $ t \le \mathbb{E}(X)$,
\begin{equation}\label{CBL}
\mathbb{P}\bigl(X \le \mathbb{E}(X) - t\bigr) \le \exp\!\left(-\frac{t^2}{2\,\mathbb{E}(X)}\right),
\end{equation}
while for each $t > 0$,
\begin{equation}\label{CBR}
\mathbb{P}\bigl(X \ge \mathbb{E}(X) + t\bigr) \le \exp\!\left(-\frac{t^2}{t + 2\,\mathbb{E}(X)}\right).
\end{equation}
\end{lemma}

Note that both inequalities above remain valid (after minor adjustments in formulation) even when only suitable bounds on $\mathbb{E}(X)$ are specified, rather than its exact value.

Recall that a graph $G=(V,E)$ is $k$-degenerate if it admits a vertex ordering such that each $v\in V$ has at most $k$ neighbours that are earlier in the ordering. Consequently, a simple greedy colouring implies a well-known upper bound for the chromatic number $\chi(G)\leq k+1$ of every such $G$.
We shall make use of the concept of a \emph{$d$-core} of a graph $G$, which we define as in~\cite{BollobasKCore,LuczakKCore} to be the maximal subgraph of $G$ with minimum degree at least $d$ (whereas it is sometimes understood in other sources as any connected component of the specified subgraph).
It is known and straightforward to verify that such subgraph is well defined, possibly empty, and can be obtained from $G$ by a simple algorithm which repeatedly removes minimum-degree vertices from $G$, together with their incident edges, as long as the remaining graph has a vertex with degree below $d$. 
Note that the set of removed \emph{edges} induces a $(d-1)$-degenerate subgraph of $G$, which is certified by the reversed ordering according to which these vertices were removed from $G$ (supplemented by placing the rest of the vertices at the beginning of the ordering). 
Consequently, we obtain the following conclusion. 
\begin{observation}\label{deg_partition}
For every integer $d \ge 0$, each graph $G$ can be edge decomposed into a $(d-1)$-degenerate subgraph $D$ and a subgraph $H$ which is empty or has minimum degree at least $d$, such that each edge of $D$ is incident with at most one vertex of $H$.
\end{observation}

We shall also make use of the following straightforward consequence of Observation 29 from~\cite{PP}, which 
also appears, yet without proof, in~\cite{deWerra}.  
\begin{corollary}[\cite{deWerra,PP}]\label{EqualDecompositionCorollary}
For every positive integer $k$ and a graph $G$ there is an edge decomposition of $G$ into $G_1,\ldots,G_k$ 
such that 
\[
\frac{d_G(v)}{k}-2 \le d_{G_i}(v) \le \frac{d_G(v)}{k}+2
\]
for every $v\in V$ and $i\in [k]$.
\end{corollary}

\section{Deterministic upper bounds for the conflict-free chromatic index}\label{DeterministicSection}

One of the key ingredients of the proof of our main result is an upper bound for the investigated parameters linear in the chromatic number of a graph,  $\chi(G)$.
In~\cite{MPK} it was already shown that $\chi'_{ccf}(G)\leq 3\log_2 (\chi(G))+4$.
However, an analogous result in the open neighbourhood setting was not known and appeared more challenging to obtain.
Below, we provide a corresponding bound for the most restrictive variant of the conflict-free chromatic index introduced above, $\chi'_{\rm cf}(G)$, building on a carefully refined idea from~\cite{MPK}.
Along the way, we also isolate several auxiliary observations that shall be useful in the main argument. To this end, we first introduce the necessary structural notions.

Let $G=(V,E)$ be a bipartite graph. Any partition $V=X\cup Y$ such that $X$ and $Y$ are independent sets in $G$ is called its \emph{vertex bipartition}. Given a subgraph $H$ of $G$ with a fixed vertex bipartition $V=X\cup Y$, we say a vertex $v \in Y$ is \emph{covered} (by $H$) if $d_{H}(v)\ge 1$, and we call it \emph{multiply-covered} if $d_{H}(v)\ge 2$. Furthermore, we say a subset of vertices $Y^{\prime}\subseteq Y$ is \emph{covered} if each vertex in $Y^{\prime}$ is covered.

A connected graph is called \emph{nontrivial} if it has at least two edges; otherwise it is called \emph{trivial}. 
A complete bipartite graph $K_{1,m}$ is called a \emph{star}, and we refer to its vertex of degree $m$ as the \emph{center}. We call the star \emph{nontrivial} if $m\ge 2$, i.e., when it is a nontrivial graph. Finally, a \emph{subdivided star} is a graph which arises from a star by a single subdivision of every edge, i.e., substituting every edge with a path of length 2. We call it a \emph{nontrivial subdivided star} if it arises from a nontrivial star,
that is when it has at least four edges.

\begin{observation}\label{idk}
Every connected nontrivial bipartite graph $G = (V, E)$ with a vertex bipartition $V = X \cup Y$ which is not a star centered at $Y$ has a subgraph $H$ which covers $Y$ such that every connected component of $H$ is either a nontrivial star centered at $X$ or a nontrivial subdivided star centered at $Y$.
\end{observation}

\begin{pf}
Starting with $H = G$, $X^{\prime} = X$ and $X^{\prime\prime} = \emptyset$, we shall repeatedly discard some edges and vertices from $H$ and update $X^{\prime}, X^{\prime\prime}$ in order to end up with a desired subgraph $H$ of $G$ with vertex bipartition $X^{\prime} \cup Y$. To that end, we repeat the following steps.

If there is a vertex $v \in X^{\prime}$ of degree $1$ in the current subgraph $H$, then we delete it from $H$ together with its incident edge, add $v$ to $X^{\prime\prime}$, and set $X^{\prime} = X \setminus X^{\prime\prime}$. Note that afterwards $H$ still covers $Y$, as $G$ is not a star centered at $Y$.

If in $X^{\prime}$ there are no more vertices of degree $1$, then every vertex in $X^{\prime}$ has degree at least $2$ in $H$. In this case, if there exists a vertex $v \in X^{\prime}$ that is adjacent in $H$ exclusively to multiply-covered vertices in $Y$, we again delete it from $H$ together with all its incident edges, add $v$ to $X^{\prime\prime}$, and set $X^{\prime} = X \setminus X^{\prime\prime}$. Note that afterwards all vertices in $X^{\prime}$ still have degrees at least $2$ in $H$, and $H$ continues to cover $Y$.

If there are no more vertices in $X^{\prime}$ with exclusively multiply-covered neighbourhoods left, we choose an arbitrary vertex $v \in X^{\prime}$ with $d_H(v) \ge 3$ and at least one multiply-covered neighbour $u \in Y$, if any such $v$ exists, and we delete the edge $uv$ from $H$. Note that afterwards, all vertices in $X^{\prime}$ still have degrees at least $2$ in $H$, and $H$ still covers $Y$. We repeat this procedure as long as there is any such vertex $v$ in $X^{\prime}$ left.

We claim that at the end of this process, we obtain $H$ consistent with our requirements. As argued, $H$ surely covers $Y$. By construction, any vertex $v \in X^{\prime}$ of degree at least $3$ in $H$ has no multiply-covered neighbours, and thus the edges incident with $v$ in $H$ form a connected component of $H$ which is a nontrivial star. Analogously, the same holds for any vertex $v \in X^{\prime}$ of degree $2$ in $H$ which has no multiply-covered neighbours. Note that by our construction, every remaining vertex $v \in X^{\prime}$ has degree $2$ in $H$ and exactly one multiply-covered neighbour. Therefore, it is easy to see that all multiply-covered vertices in $Y$ are centers of nontrivial subdivided stars, making up the remaining connected components of $H$.
\qed
\end{pf}

\begin{observation}\label{3bi}
Every bipartite graph $G = (V, E)$ with no trivial connected components and vertex bipartition $V = X \cup Y$ admits a partial edge $3$-colouring satisfying all edges of $G$ such that every vertex $v \in Y$ is incident with a uniquely coloured edge, and every vertex $u \in X$ is incident with a uniquely coloured edge whenever at least one edge in $E(u)$ is coloured.
\end{observation}

\begin{pf}
We may assume $G$ is nontrivial and connected. If $G$ is a star centered at $Y$, it is enough to colour one of its edges with $1$, another one with $2$, and the remaining edges, if any, with $3$. So we may assume in what follows that $G$ is not a star centered at $Y$.

Let $H$ be a subgraph of $G$ given by Observation~\ref{idk}, with vertex bipartition $X^{\prime} \cup Y$. For every component of $H$ being a star centered at $X$, we colour one of its edges with $1$, another one with $2$, and the remaining edges, if any, with $3$. For every component of $H$ being a subdivided star centered at  
$y \in Y$, we colour one edge incident to $y$ with $3$, all other edges incident to $y$ with $2$, and the remaining edges (not incident to $y$) with $1$.

Then it is easy to see that every vertex in $X^{\prime} \cup Y$ is incident with a uniquely coloured edge 
in the given partial colouring of $G$. It is also straightforward to verify that all coloured edges (i.e., those in $H$) are satisfied. 

Now consider any uncoloured edge $xy \in E$ with $x \in X$ and $y \in Y$. If $x \notin X^{\prime}$, then no coloured edge is incident to $x$, while $y$ is incident with a uniquely coloured edge; hence, $xy$ is satisfied. Thus, we may assume that $x \in X^{\prime}$. By the colouring construction, $x$ is incident with two uniquely coloured edges. Therefore, since $xy$ is uncoloured, it could be unsatisfied only if $y$ is incident with at least two colours. This occurs only when $y$ is a center of a subdivided star in $H$, in which case it is incident with colours $3$ and $2$. However, by the colouring construction, every vertex in $X^{\prime}$, and in particular $x$, is incident with a uniquely coloured edge coloured $1$. Thus, all edges in $G$ are satisfied.
\qed
\end{pf}

Extending the partial edge $3$-colouring given by Observation~\ref{3bi} by colouring all remaining edges with one additional colour yields the following Corollary~\ref{4Bi} for $\chi'_{\rm ocf}(G)$; 
the corresponding result for the closed setting was obtained earlier in~\cite{MPK}. 

\begin{corollary}\label{4Bi}
For every bipartite graph $G$, we have $\chi'_{\rm ocf}(G) \le \chi'_{\rm cf}(G) \le 4$. 
\end{corollary}

\begin{corollary}\label{4Bi16}
Every bipartite graph $H = (V, F)$ without trivial connected components admits a partial edge $4$-colouring satisfying all edges such that every vertex in $V$ is incident with a uniquely coloured edge.
\end{corollary}

\begin{pf}
Consider a partial edge colouring of $H$ with colours $1, 2, 3$ given by Observation~\ref{3bi}.
This colouring ensures that all  edges of $H$ are satisfied. 
Moreover, every vertex in $Y$, as well as every vertex in $X$ that is incident with at least one coloured edge, is incident with a uniquely coloured edge. Let us denote the set of such vertices in $X$ by $X'$, and set $X^{\prime\prime} = X \setminus X^{\prime}$.

For every vertex $x \in X^{\prime\prime}$, choose an arbitrary single edge joining it to $Y$ in $H$ and colour it with $4$. Consequently, all edges in $H$ remain satisfied, while every vertex in $V$ becomes incident with a uniquely coloured edge. 
\qed
\end{pf}

\begin{observation}\label{16m}
Every graph $G = (V, E)$ 
which can be edge-decomposed into a spanning bipartite subgraph $H'$ without trivial connected components 
and a subgraph whose edges form a matching $M'$ admits a partial edge $16$-colouring satisfying all edges of $G$.
\end{observation}

\begin{pf}
Let $H = (V, F)$ be a maximal spanning bipartite subgraph of $G$ without trivial connected components that contains the edges of $H'$, and let $V = X \cup Y$ be its vertex bipartition. 
Set $M = E \setminus F$. Since $M\subseteq M'$, the edges of $M$ form a matching.
Moreover, by the maximality of $H$, the set $M$  contains no edge $xy$ with $x \in X$ and $y \in Y$. 

Let $c: E^{\prime} \to [4]$, where $E^{\prime} \subseteq F$, be a partial edge $4$-colouring of $H$ guaranteed by
Corollary~\ref{4Bi16}. 
Observe that any edges of $G$ that are unsatisfied by $c$, if such exist, must belong to $M$.
Fix also an auxiliary proper vertex colouring $\omega: V \to [2]$ of the graph $(V, M)$.
We now define a new partial edge $16$-colouring $C: E^{\prime} \to [4] \times [2] \times [2]$ as follows. For every 
edge $xy \in E^{\prime}$ with $x \in X$ and $y \in Y$, set $$C(xy) := (c(xy), \omega(x), \omega(y)).$$

Note that every edge $e$ in the bipartite subgraph $H$ remains satisfied under the colouring $C$, as its first coordinate $c(e)$ inherently guarantees a unique colour in the neighbourhood of $e$.

Now consider any edge $xx'\in M$ where $x,x'\in X$.
By Corollary~\ref{4Bi16}, there is an edge $e$ in $E'(x)$ whose colour $c(e)$ is unique within $E'(x)$. 
It is straightforward to verify that $C(e)$ is then unique in the neighbourhood of $xx'$: its first coordinate ensures uniqueness among edges incident with $x$, while the second coordinate distinguishes it from all colours  $C(e')$ with $e'\in E'(x')$. 

An analogous argument, using the third coordinate, shows that every edge in $M$ with both endpoints in 
$Y$ is also satisfied. Consequently, all edges of $G$ are satisfied by $C$.
\qed
\end{pf}

\begin{theorem}\label{3logtemp}
Every graph $G$ with chromatic number at most $\alpha$ admits a partial edge colouring using at most $3\lceil \log_2 \alpha \rceil$ colours such that all edges of $G$ are satisfied, except possibly for a set $M$ of uncoloured edges forming a matching.
\end{theorem}

\begin{pf}
We prove the theorem by induction on $\alpha$. The assertion is trivial for $\alpha=1$, while for $\alpha=2$ it follows from Observation \ref{3bi}. Thus, assume that $\alpha \ge 3$ and $G$ is a connected graph with at least two edges.

Set $\beta = 2^{\lceil \log_2 \alpha \rceil}$; then $\lceil \log_2 \alpha \rceil = \log_2 \beta$. Colour the vertices of $G$ properly with colours in $[\beta]$, allowing some colour classes to be empty. Let $X$ be the set of vertices coloured with colours in $[\beta/2]$, and let $Y=V\smallsetminus X$.  
Choose such a colouring so as to maximise the number of edges between $X$ and $Y$, i.e., $|E[X,Y]|$.

Observe that the bipartite graph $G[X,Y]$ cannot contain an isolated edge. Indeed, suppose that there exists an edge $xy\in E[X,Y]$ with $x\in X$ and $y\in Y$ that is isolated in $G[X,Y]$.
Since $G$ is connected and has at least two edges, this edge must be adjacent to another edge, say  $xx^{\prime} \in E[X]$.
However, swapping the colours of $x$ and $y$ preserves the properness of the colouring while increasing
$|E[X,Y]|$, contradicting the maximality of our choice.

Since the chromatic numbers of $G[X]$ and $G[Y]$ are at most $\beta/2 < \alpha$, we may apply the induction hypothesis to both subgraphs.
Thus, we obtain partial edge colourings of $G[X]$ and $G[Y]$ using the same set of $3 \log_2(\beta/2) = 3 \log_2 \beta  -  3$ colours, such that all edges in $E[X] \cup E[Y]$ are satisfied, except possibly for a set $M$ of uncoloured edges forming a matching.
Finally, by Observation~\ref{3bi}, it suffices to use $3$ additional colours on edges in $E[X,Y]$ to ensure that all these edges are also satisfied.
\qed
\end{pf}

\begin{theorem}\label{3log}
Every graph $G = (V, E)$ with chromatic number at most $\alpha$ admits a partial edge colouring with $3\lceil \log_2 \alpha \rceil + 16$ colours satisfying all non-isolated edges of $G$.
\end{theorem}

\begin{pf}
We may assume $G$ is connected and nontrivial. Fix any spanning tree $T = (V, F)$ of $G$, and let $G^{\prime}$ be the graph obtained from $G$ by removing the edges of $T$. Clearly, $\chi(G^{\prime}) \le \alpha$.
By Theorem \ref{3logtemp}, we may partially colour $G^{\prime}$ with no more than $3\lceil \log_2 \alpha \rceil$ colours in order to satisfy all edges of $G^{\prime}$ except possibly for a set $M$ of uncoloured edges forming a matching.

Consider now the graph $G^{\prime\prime} := (V, F \cup M)$. This graph satisfies the assumptions of Observation~\ref{16m}, and hence we may use a new set of $16$ colours to colour some of its edges so that all edges of $G''$ get satisfied. Note that this does not affect the edges of $G^{\prime} \setminus M$, which remain satisfied as well.
\qed
\end{pf}

As mentioned, Theorem~\ref{3log} shall serve as one of the key tools in our probabilistic approach below. However, similarly to Observation~\ref{3bi}, it also yields the following deterministic consequence, which may be of independent interest. See~\cite{MPK}  for a slightly stronger result in the closed neighbourhood setting.
\begin{corollary} \label{3logfull}
For every graph $G$, we have $\chi_{\rm ocf}^{\prime}(G) \leq \chi_{\rm cf}^{\prime}(G) \leq 3 \lceil\log_2 \chi(G)\rceil + 17 < 3 \log_2 \chi(G)+20$. 
\end{corollary}

\section{Proof of Theorem~\ref{GoodAsymptCFE-thm}}

\subsection{General overview of the approach}

Fix $\varepsilon>0$ and let $G=(V,E)$ be a graph with maximum degree $\Delta$. We may assume that $G$ contains no isolated edges. Whenever necessary we shall also assume that $\Delta$ is large enough;
in particular, we do not attempt to optimise or explicitly determine the corresponding threshold $\Delta_\varepsilon$.

Our vague primary idea relies on edge-decomposing our graph into two subgraphs $H$ and $D$. 
The subgraph $D$ shall have upper-bounded chromatic number, which shall allow us to apply Theorem~\ref{3log} to partially colour its edges using a small number of $o(\log_2\Delta)$ colours, thereby satisfying all of its edges except for a small exceptional set.

In turn, the subgraph $H$ shall be required to have sufficiently large minimum degree relative to its maximum degree, which shall be convenient for the application of the probabilistic method.
The corresponding central part of the argument shall require the bulk of about $\log_2\Delta$ colours to satisfy 
the majority of the edges incident with any given vertex $v$, namely all except at most   $2\log_2\Delta$ of these. 
Consequently, the remaining unsatisfied edges shall span a graph of small maximum degree, which can be handled using an additional $o(\log_2\Delta)$ of colours. 
Throughout the construction, we shall use disjoint sets of colours at different stages in order to avoid unwanted interactions between the respective colourings.

One of the issues we shall need to additionally resolve concerns several problematic matchings that may emerge in the course of the argument and whose edges may remain unsatisfied after the main stage.
To address this, we shall initially additionally extract a forest $F$ with bounded degrees from $G$. In the final stage of the proof, we shall then demonstrate that $F$, along with the mentioned matchings, satisfies the assumptions of Observation~\ref{16m}, which allows us to use at most 16 additional colours to complete a partial colouring that satisfies all edges of $G$. The details follow.

\subsection{Initial decomposition}

Set 
\begin{equation}\label{dValue}
d:=\left\lfloor 10(\log_2\Delta)^3\right\rfloor.
\end{equation}
We first apply Observation~\ref{deg_partition} to $G$, thus edge-decomposing it into its $d$-core $H'$, with minimum degree at least $d$, and a graph $D'$ of degeneracy at most $d-1$. 
Next, we apply Corollary~\ref{EqualDecompositionCorollary} to $H'$ to obtain its temporary auxiliary edge-decomposition into $H'_1,H'_2,H'_3$ with $d_{H'_i}(v)\geq d/3-2$ for every $v\in V(H')$, i.e. with $\delta(H'_i)\geq d/3-2$ for $i\in [3]$.  We then fix any spanning forest $F'$ of $H'_3$, which by degree conditions above is a spanning forest of $H'$ without isolated edges. Define $H:=H'-E(F')=(V(H'),E(H')-E(F'))$.
Let $J$ be the matching consisting of all isolated edges of $D'$, if any.
Since $G$ has no isolated edges, every edge in $J$ has exactly one endpoint in $V(H')$. 
Consequently, the graph $F:=F'+J=(V(H')\cup V(J),E(F')\cup J)$ is a forest. Finally, set $D:=D'-J$. 
Thus, we obtain an edge decomposition of  of  $G$ into $H$, $D$ and $F$ such that:
\begin{description}
\item[(i)] $\delta(H)\geq 0.65d$  
and $\Delta(H)\leq \Delta$; 
\item[(ii)] $D$ has no trivial components and is $(d-1)$-degenerate; in particular, $\chi(D)\leq d$;
\item[(iii)] $F$ is a forest without trivial components and $V(H)\subseteq V(F)$.
\end{description}

\subsection{Outline of the random colouring procedure}

Suppose that we randomly partition $V(H)$ into three subsets $U_1, U_2$, and $U_3$ so that each of $U_1$ and $U_2$ contains roughly half of the vertices of $H$, while $U_3$ is relatively small, yet large enough to ensure that every vertex $v \in U_2$ 
has some edges joining it with $U_3$. 
For each such vertex $v$ fix one such edge and denote it by $e_v$. 
We then colour all edges $e_v$, for $v \in U_2$, 
with a single colour, say $\alpha_1$. In this way, every edge between $U_1$ and $U_2$ becomes satisfied by $\alpha_1$. By randomness, we may expect that this procedure satisfies roughly half of the edges incident with each vertex $v$ in the set $U_1 \cup U_2$, which comprise the vast majority of vertices of $H$. 

Now suppose, hypothetically, that this procedure could be repeated independently about
$\log_2 \Delta$ times, each time generating a new partition of $V(H)$ and using a new colour $\alpha_i$. 
Since in each round approximately half of the edges incident with a typical vertex $v$ gain a unique colour in their neighbourhoods, after $\log_2 \Delta$ rounds we would expect that only a constant number of edges incident with $v$, roughly $(1/2)^{\log_2 \Delta} d(v)$, remain unsatisfied.
This suggests that the number of such remaining edges would be sufficiently small for our approach to succeed.
There are, however, several issues that must be resolved.

First, in order to complete the colouring after the process outlined above, all coloured edges must already be satisfied. To achieve this, we shall need to use four additional colours within each set $U_3$. However, we cannot afford to introduce five new colours in every round, as this would result in a total of roughly $5 \log_2 \Delta$ colours. To overcome this difficulty, while the sets $U_1$ and $U_2$ shall be chosen anew in each round, the set $U_3$, together with the four colours used within it, shall remain fixed over 
blocks of $k$ consecutive rounds, where $k$ shall be chosen to be a sufficiently large constant.

Second, we have not yet ensured that the sets of edges coloured in different rounds are disjoint. 
To address this, we first decompose $H$ into approximately  $\log_2 \Delta/k$ edge-disjoint subgraphs $H_i = (V(H), E_i)$. Then, in each block of $k$ consecutive rounds, we select edges to be coloured from a distinct set $E_i$. 
To make the argument work, we shall need to ensure that both the partitions and the decompositions satisfy certain edge-distribution balancing conditions; the details are provided below.

\subsection{Further decomposition and random vertex partitions}

Set
\begin{equation}\label{ksValue}
k = \left\lceil \frac{5}{\varepsilon} \right\rceil,~~~~ s = \left\lceil \frac{\log_2 \Delta}{k} \right\rceil,~~~~{\rm and}~~~~ U=V(H).
\end{equation}
Note that $ks\leq \log_2\Delta+C(\varepsilon)$, where $C(\varepsilon)$ is a constant depending only on $\varepsilon$.

By Corollary~\ref{EqualDecompositionCorollary}, we may edge-decompose $H$ into $s$ spanning subgraphs $H_i=(U,E_i)$, $i\in [s]$, such that for every vertex $v\in U$ and each $i\in[s]$, by~(\ref{dValue}), (i) and~(\ref{ksValue}), we have: 
\begin{equation}\label{dHi}
d_{H_i}(v)\geq \frac{d_{H}(v)}{s}-2\geq 6(\log_2\Delta)^2.
\end{equation}

To define the random partitions, for all vertices $v\in U$, we introduce independent random variables $X_{v,i,j}$, for $i\in [s]$ and $j \in [k]$, given by:
\begin{equation}\label{Zdef}
X_{v,i,j}=\begin{cases}
1 & \text{with probability } 1/2, \\ 
2 & \text{with probability } 1/2,
\end{cases}
\end{equation} 
and additional independent random variables $Z_v$ taking values in $[s]$ uniformly at random.
These variables define, for every $i\in [s]$ and $j \in [k]$, a partition of $U$ into three sets:
\begin{align*}
 U_1^{(i,j)} &=\{v\in U: X_{v,i,j}=1 ~\text{ and }~ Z_v \neq i\},\\
 U_2^{(i,j)} &=\{v\in U: X_{v,i,j}=2 ~\text{ and }~ Z_v \neq i\},\\
 U_3^{(i,j)} &=U_3^i= \{v\in U: Z_v = i\}.
\end{align*}
Note that $U_3^{(i,j)}$ depends only on $Z_v$ and hence is the same across all $k$ partitions corresponding to a fixed $H_i$, thus we may abbreviate its notation to $U_3^i$.

To carry out the colouring construction, we shall require that every vertex has sufficiently many edges joining it with $U_3^i$ within each $H_i$.
For every vertex $v\in U$ and $i\in[s]$, we thus set $N_{v,i}=|N_{H_i}(v)\cap U_3^i|$ and denote the following undesired event:
$$Q_{v,i}:~~N_{v,i} <\log_2\Delta.$$ 
By~(\ref{ksValue}) and~(\ref{dHi}),
$$\mathbb{E}\left(N_{v,i}\right)\geq \frac{d_{H_i}(v)}{s} \geq 5\log_2\Delta.$$
Since the variables $Z_u$, for $u\in N_{H_i}(v)$, are independent, the Chernoff Bound (\ref{CBL}) yields:
\begin{equation}\label{Qvi-bound}
\mathbb{P}\left(Q_{v,i}\right)
\leq \exp\left(-\frac{\left(4\log_2\Delta\right)^2}{10\log_2\Delta}\right)
= \Delta^{-\frac{1.6}{\ln 2}}
< \Delta^{-2.2}.
\end{equation}

Next, we bound the probability of another undesirable event: that a given vertex is incident with many edges that never cross between $U_1^{(i,j)}$ and $U_2^{(i,j)}$, and hence may remain unsatisfied throughout the entire process.
For any given vertex $v\in U$, its neighbour $u\in N_H(v)$ and arbitrary $i\in[s]$ and $j\in [k]$, let us denote the event:
$$C_{vu,i,j}:~~\left(v\in U_1^{(i,j)} \text{ and } u\in U_2^{(i,j)}\right) \quad \text{or} \quad \left(v\in U_2^{(i,j)} \text{ and } u\in U_1^{(i,j)}\right).$$
If this event holds, we say $vu$ is in a \emph{good configuration} in partition $j$ of subgraph $H_i$.
Otherwise, $vu$ is in a \emph{bad configuration} in this partition.
Note this notation applies to any edge $uv$ of $H$ (not only to edges in the given $H_i$).
Denote further an event indicating that $vu$ is in bad configurations in all partitions:
\begin{equation}\label{Cvu-definition}
C_{vu}:~~ \bigwedge_{i\in[s],j\in[k]}\overline{C_{vu,i,j}}.
\end{equation}
We say that $vu$ \emph{is always bad} if this event occurs.

Observe that if $v\in U_3^i$ or $u\in U_3^i$, i.e. when $Z_v=i$ or $Z_u=i$, then $C_{vu,i,j}$ cannot occur.
Otherwise, $C_{vu,i,j}$ holds with probability $1/2$. 

Notice moreover that regardless of the values of $Z_v$ and $X_{v,i',j'}$ with $i'\in[s]$, $j'\in[k]$ (i.e. for any fixed evaluations of these), the probability that $Z_u=Z_v$ equals exactly $1/s$, whereas for any $i\neq Z_v,Z_u$, the probability that $C_{vu,i,j}$ holds remains exactly $1/2$ for every $j\in[k]$. Therefore,
\begin{equation}\label{Cuv-estimate}
\mathbb{P}\left(C_{vu}\right) = \frac{1}{s} \left(\frac{1}{2}\right)^{(s-1)k} 
+ \frac{s-1}{s} \left(\frac{1}{2}\right)^{(s-2)k} 
\leq \left(\frac{1}{2}\right)^{(s-2)k} \leq \frac{2^{2k}}{\Delta},
\end{equation}
and for the same reasons, the events $C_{uv}$ corresponding to all neighbours $u\in N_H(v)$ of $v$ are independent and have the same probability, estimated above in~(\ref{Cuv-estimate}). 
Thus, denoting by 
\[
B_v = \bigl| \{u\in N_{H}(v): vu \text{ is always bad}\} \bigr|.
\]
we may use the Chernoff Bound to bound the probability of the following event:
\[
A_v: \quad B_v > 2\log_2\Delta.
\]
As by~(\ref{Cuv-estimate}), $\mathbb{E}(B_v)\leq 2^{2k}$, by~(\ref{CBR}) we obtain:
\begin{equation}\label{Av-bound}
\mathbb{P}\left(A_v\right)
\leq \exp\left( - \frac{(2\log_2\Delta - 2^{2k})^2}{(2\log_2\Delta - 2^{2k})+2^{2k+1}} \right)
\leq \Delta^{-2.5},
\end{equation}
for sufficiently large $\Delta$. 

Note that the bound on $\mathbb{P}(A_v)$ could alternatively be formalised using the law of total probability, conditioning on fixed values of $Z_v$ and  $X_{v,i',j'}$; however, as discussed, such a lengthy justification is unnecessary.

Finally, observe that the events $Q_{v,i}$ and $A_v$ depend only on random choices made for vertices at distance at most $1$ from $v$. Hence, each such event is mutually independent of all other events associated with vertices $u$ at distance greater than $2$ from $v$. For each vertex, there are $s$ events of type $Q_{v,i}$
and one event of type $A_v$. Thus, each such event is mutually independent of all but at most $\Delta^2(s+1)$ other events. 
By~(\ref{Qvi-bound}) and~(\ref{Av-bound}), the probability of each of these bad events is at most $\Delta^{-2.2}$. 
Thus, since $e \cdot \Delta^{-2.2} \cdot (\Delta^2(s+1)+1) < 1$ for sufficiently large $\Delta$, Lemma \ref{LLL} guarantees that with positive probability none of the events  $Q_{v,i}$ or $A_v$ occurs.

Consequently, there exist $sk$ vertex partitions $U= U_1^{(i,j)}\cup U_2^{(i,j)}\cup U_3^{(i,j)}$ for $i\in[s]$ and $j\in[k]$ such that $U_3^{(i,1)}=U_3^{(i,2)}=\ldots=U_3^{(i,k)} = U_3^{i}$ for every $i\in [s]$, 
where $U_3^{1}\cup U_3^{2}\cup\ldots\cup U_3^{s}$ is a partition of $U$, and:
\begin{description}
\item[$(1^\circ)$] every vertex $v\in U$ is joined by at least $\log_2\Delta$ edges with $U_3^i$ in $H_i$, for each $i\in[s]$;
\item[$(2^\circ)$] every vertex $v\in U$ is incident in $H$ with at most  $2\log_2\Delta$  edges $vu$ which are always in a bad configuration, across all the $sk$ partitions, i.e. such for which~(\ref{Cvu-definition}) holds.
\end{description}
We fix any collection of such $sk$ partitions.

\subsection{Random partition-based edge colouring}

To satisfy the vast majority of edges in $H$, we now describe the procedure for assigning colours to selected edges joining $U_2^{(i,j)}$ with $U_3^{(i,j)}$ in $H_i$, for $i\in [s]$ $j\in[k]$.
We process each subgraph $H_i$ independently, one by one. Initially no edges are coloured.

Fix a subgraph $H_i$, we iteratively process its $k$ partitions. In the $j$-th step, corresponding to the partition $U_1^{(i,j)} \cup U_2^{(i,j)} \cup U_3^{(i)}$,  we introduce a fresh, dedicated colour $c_{i,j}$. For every vertex $v \in U_2^{(i,j)}$, we arbitrarily pick one previously uncoloured edge $e$ joining $v$ with its neighbour $u \in U_3^{(i)}$ in $H_i$ and colour it with $c_{i,j}$. Such uncoloured edge $e$ always exists, as by  $(1^\circ)$, each vertex $v$ has at least $\log_2 \Delta$ neighbours in $U_3^{(i)}$, while at most $k-1$ of its incident edges in $H_i$ could have been coloured in the preceding steps. 

This procedure ensures that every edge $e\in E(H)$ which is in a good configuration in partition $j$ of the subgraph $H_i$ becomes instantly satisfied by the colour $c_{i,j}$, and remains so, since each step uses a distinct colour.

However, the coloured edges might not themselves be satisfied. 
To address this, we additionally colour certain edges within the sets $U_3^i$.
Fix $i\in [s]$. 
By~$(1^\circ)$, every vertex $v\in U_3^i$ has at least $\log_2\Delta$ neigbours in $U_3^i$ within $H_i$. 
Thus, each connected component of $H_i[U_3^i]$ is nontrivial, hence the same holds for arbitrarily chosen spanning forest $F_i$ of $H_i[U_3^i]$. 
We may therefore apply Corollary~\ref{4Bi16} to $F_i$ in order to colour its edges with $4$ fresh colours so that all newly coloured edges get satisfied and every vertex in $U_3^i$ becomes incident with a uniquely coloured edge in $H_i[U_3^i]$. 
Consequently, every edge coloured with $c_{i,j}$, for $j\in [k]$, becomes satisfied, since it has exactly one endpoint in $U_3^i$.

We repeat this procedure for each $i\in [s]$, using $4$ additional colours per set $U_3^i$.

Altogether, this yields a partial edge colouring of $H$ using at most
\begin{eqnarray}
sk+4s &\leq& (k+4) \left(\frac{\log_2\Delta}{k}+1\right) = \log_2\Delta + 4 \frac{\log_2\Delta}{k} +k+4\nonumber\\ 
&\leq& \log_2\Delta + \frac{4}{5}\varepsilon \log_2\Delta +k+4 \leq  \log_2\Delta + \frac{4}{5}\varepsilon \log_2\Delta + \frac{1}{20}\varepsilon \log_2\Delta \label{NoHCol}
\end{eqnarray} 
colours and all coloured edges are satisfied, as well as every edge of $H$ that is not always bad. 
Let $R$ be the subgraph of $H$ induced by the edges that remain unsatisfied at this stage. Since every such edge must be always bad, by $(2^\circ)$, we have
\begin{equation}\label{dRv}
d_R(v)\leq 2\log_2\Delta
\end{equation}
for every vertex $v\in U$.

\subsection{Satisfying the remaining edges}

Let $M_R$ be the matching consisting of all isolated edges of $R$, if any, and set $R':=R-M_R$.
Since $H$, $D$ and $F$ form an edge decomposition of $G$, every edge of $G$ that remains unsatisfied at this stage belongs to one of the graphs $R$, $D$ or $F$, and hence to one of $R'$, $D$ or $F+M_R$. 
We now colour edges in each of  the latter  three graphs separately so as to satisfy all of their edges.

By~(\ref{dRv}), we have $\chi(R')\leq\chi(R)\leq\Delta(R)+1\leq 2\log_2\Delta+1$. Thus, we may 
apply Theorem~\ref{3log} to partially colour the edges of $R'$ with at most
\begin{equation}\label{NoR'Col}
3\left\lceil\log_2(2\log_2\Delta+1)\right\rceil+16 \leq \frac{1}{20}\varepsilon \log_2\Delta
\end{equation} 
new colours, so that all non-isolated edges of $R'$ get satisfied. Since $R'$ has no isolated edges, all edges of $R'$ are now satisfied.

Similarly, by~(ii), the graph $D$ has no isolated edges and satisfies $\chi(D)\leq d$. 
Hence, by Theorem~\ref{3log}, we may partially colour the edges of $D$ using at most
\begin{equation}\label{NoDCol}
3\left\lceil\log_2 d \right\rceil+16 \leq \frac{1}{20}\varepsilon \log_2\Delta
\end{equation} 
previously unused colours so that all edges of $D$ are satisfied.

Finally, observe that by construction all vertices incident with edges in $M_R$ belong to $V(H)$. Since, by (iii), $F$ is a forest without trivial components and  $V (H) \subseteq V (F)$, we may apply Observation~\ref{16m} to the graph $F+M_R$. This allows us to partially colour its edges using at most 16 additional colours so that all edges of $F+M_R$ are satisfied.

At this point, all edges of $G$ are satisfied.
Assigning one additional fresh colour to all  yet uncoloured edges completes the construction. 
By~(\ref{NoHCol}), (\ref{NoR'Col}) and~(\ref{NoDCol}), the total number of colours used is at most
$$ \left(\log_2\Delta + \frac{4}{5}\varepsilon \log_2\Delta + \frac{1}{20}\varepsilon \log_2\Delta\right) + \frac{1}{20}\varepsilon \log_2\Delta+ \frac{1}{20}\varepsilon \log_2\Delta +16+1
\leq (1+\varepsilon)\log_2\Delta,$$
as required. \qed

\section{Lower bounds via random graphs and concluding remarks}
\label{SectionRandomComments}

By Theorem~\ref{GoodAsymptCFE-thm}, we immediately obtain the following.
\begin{corollary}\label{GoodAsymptCFE-cor}
For every graph $G$ with maximum degree $\Delta>1$,
$$
\chi'_{\rm ccf}(G)\leq (1+o(1))\log_2\Delta~~~{\rm and}~~~\chi'_{\rm ocf}(G)\leq (1+o(1))\log_2\Delta.
$$
\end{corollary}

In this section, we show that $(1-o(1))\log_2\Delta$ colours are typically necessary for conflict-free edge colourings, both in the closed and open settings. To this end, we consider the standard binomial random graph
$\mathbb{G}_{n,p}$, where $0\leq p\leq 1$ may depend on $n$, see e.g.~\cite{gnp}.

We first show that relatively large number of colours is required with high probability whenever $p\gg (\log_2n)^2/n$, that is, when $pn/(\log_2n)^2 \rightarrow \infty$ as $n \rightarrow \infty$.
The second part of the argument extends approaches used in~\cite{DebskiPrzybylo} and~\cite{KPreg2}  for complete graphs. We discuss the implications of the theorem after its proof.

\begin{theorem}\label{RandomGraphsTech}
If $p\gg (\log_2n)^2/n$, then a.a.s.
$$\chi'_{\rm ccf}(\mathbb{G}_{n,p})\geq \log_2(pn)-2\log_2\log_2n~~~{\rm and}~~~
\chi'_{\rm ocf}(\mathbb{G}_{n,p})\geq \log_2(pn)-2\log_2\log_2n.$$
\end{theorem}

\begin{pf}
We shall show that w.h.p., $k=\log_2n-\log_2p^{-1}-2\log_2\log_2n$ colours do not suffice in either setting. 
For convenience we omit ceiling and floor signs, to simplify presentation of the idea. Set 
$$s:=p^{-1}(\log_2n)^2$$ 
and assume that $[n]$ is the vertex set of $\mathbb{G}_{n,p}$. 
We first observe that w.h.p. 
\begin{equation}\label{Sspans}
\forall S\subseteq [n], |S|=s:~ S~{\rm spans~ at~least}~2p^{-1}(\log_2n)^3~{\rm edges}. 
\end{equation}
Indeed, for any fixed $S\subseteq [n]$ with $|S|=s$, the expected number of edges it spans is ${s \choose 2}p \geq 0.1p^{-1}(\log_2n)^4$, where we assume that $n$ is large enough, whenever needed. Thus, by the Chernoff Bound, the probability that $S$ spans less than $2p^{-1}(\log_2n)^3$ edges is at most
$$\exp\left(-\frac{\left(0.1p^{-1}(\log_2n)^4-2p^{-1}(\log_2n)^3\right)^2}{2\cdot 0.1p^{-1}(\log_2n)^4}\right) 
\leq n^{-2s}.$$
Hence, by the union bound, the probability that~(\ref{Sspans}) does not hold is bounded above by
${n \choose s}n^{-2s}\leq n^{-s}=o(1)$, as claimed.

Now fix a graph $G = ([n], E)$ satisfying~(\ref{Sspans}), and consider any edge colouring $c$ of $G$ with $k$ colours. For each vertex $v$, let $P_v$ denote its \emph{palette}, i.e., the set of colours appearing on the edges incident with $v$. Since at most $2^k=np/(\log_2n)^2=n/s$ distinct pallets may occur in $G$, at least one palette, say $\mathcal{P}$, is assigned to at least $s$ vertices; let $S$ be any subset of exactly $s$ such vertices.

By ~(\ref{Sspans}), the set $S$ spans at least $2p^{-1}(\log_2 n)^3 \geq 2ks$ edges. Hence, there exists a set $F \subseteq E[S]$ of at least $2s$ edges all coloured with the same colour $\alpha$. It follows that there is an edge $uv \in F$ such that $d_F(v) \geq 3$, since otherwise all degrees in $F$ would be at most $2$.
Thus, the colour $\alpha$ of $uv$ appears at least twice in both $E(uv)$ and $E[uv]$. 
However, any colour in $\mathcal{P}\smallsetminus\{\alpha\}$ must appear at least twice in $E(uv)\subseteq E[uv]$ as well, since $P_u=\mathcal{P}=P_v$.
Consequently, no colour is unique in either the open or closed neighbourhood of the edge $uv$, and hence $c$ is neither a closed nor an open conflict-free colouring.
\qed
\end{pf}

It is well known that w.h.p., $\Delta(\mathbb{G}_{n,p})\sim np$ for $p\gg \log_2n/n$; see e.g.~\cite{gnp}.
Therefore, Theorem~\ref{RandomGraphsTech} implies that a.a.s., 
$$\chi'_{\rm ccf}(\mathbb{G}_{n,p})\geq (1-o(1)) \log_2\Delta(\mathbb{G}_{n,p})~~~ {\rm and}~~~ \chi'_{\rm ocf}(\mathbb{G}_{n,p})\geq (1-o(1)) \log_2\Delta(\mathbb{G}_{n,p})$$
whenever $pn$ is a superpolylogarithmic function in $n$, i.e., when $p=(\ln_2n)^\omega/n$ where $\omega$ is any function such that $\omega\rightarrow\infty$ arbitrarily slowly with $n$. 
In particular, this holds for $p=n^{\varepsilon-1}$, where $0<\varepsilon\leq 1$ is any fixed constant. 

Moreover, for every $p \geq (\log_2 n)^{2+\gamma}/n$, where $\gamma > 0$ is any constant, Theorem~\ref{RandomGraphsTech} along with Corollary~\ref{GoodAsymptCFE-cor} imply that both parameters are $\Theta(\log_2 \Delta(G_{n,p}))$ with high probability.\\

Therefore, the bounds from Corollary~\ref{GoodAsymptCFE-cor} are asymptotically optimal in several respects.
In particular, they are matched from below by broad families
arising from the analysis of random graphs, and hence by asymptotically almost all graphs. 
Establishing these upper bounds required handling low-degree vertices, which introduce significant technical complications and lead to increase of the corresponding \emph{vertex} parameters, particularly in the open setting. Perhaps surprisingly, accounting for these difficulties did not affect the asymptotic order of the obtained upper bounds for the edge setting.

In view of our asymptotically precise solutions, it is natural to ask whether analogous or even stronger results can be obtained with a purely deterministic approach.
For example, one could aim to improve the results in Corollary~\ref{3logfull}, in particular with respect to the chromatic number $\chi(G)$, rather than the maximum degree of a graph.

Finally, recall that the conflict-free chromatic index of a graph can be interpreted as the conflict-free chromatic number of its line graph. From this perspective, our research may be viewed as studying the latter parameter within the class of line graphs, which exhibit a certain sparsity structure. 
This points towards a natural direction for further research, namely to investigate to what extent our results for $\chi_{\rm ccf}(G)$ and $\chi_{\rm ocf}(G)$ extend to more general sparse graph classes, such as e.g. $K_{1,r}$-free graphs, and to determine when  the relevant parameters remain closer to $\log \Delta$ rather than $(\log \Delta)^2$.


\begin{thebibliography}{99}


\bibitem{Abel2018}
Z.~Abel, V.~Alvarez, E.~D.~Demaine, S.~P.~Fekete, A.~Gour, A.~Hesterberg, P.~Keldenich, and C.~Scheffer,
Conflict-free coloring of graphs,
SIAM Journal on Discrete Mathematics, 32(4) (2018) 2675--2702.

%\bibitem{LLL} N. Alon, J.H. Spencer, The Probabilistic Method, second ed., Wiley, New York, 2000.

\bibitem{LLL} N. Alon, J.H. Spencer, The Probabilistic Method, second ed., Wiley, New York, 2000.

%\bibitem{Berge}
%C. Berge, Two theorems in graph theory, Proceedings of the National Academy of Sciences of the United States of America 43(9) (1957) 842--844.

\bibitem{Hindusi} 
Sriram Bhyravarapu, Subrahmanyam Kalyanasundaram, Rogers
Mathew,
A short note on conflict-free coloring on closed neighborhoods of bounded degree graphs,
J. Graph Theory 97(4) (2021) 553--556.

\bibitem{Hindusi3}
Sriram Bhyravarapu, Subrahmanyam Kalyanasundaram, Rogers
Mathew, Conflict-Free Coloring Bounds on Open Neighborhoods,
Algorithmica 84(8) (2022) 2154--2185.

\bibitem{Hindusi2} 
Sriram Bhyravarapu, Subrahmanyam Kalyanasundaram, Rogers
Mathew, Conflict-free coloring on claw-free graphs and interval graphs.
In 47th International Symposium on Mathematical Foundations of Computer
Science, MFCS 2022, August 22-26, 2022, Vienna, Austria, volume
241 of LIPIcs, pages 19:1--19:14. Schloss Dagstuhl - Leibniz-Zentrum
f\"ur Informatik, 2022. https://doi.org/10.4230/LIPIcs.MFCS.2022.19
doi:10.4230/LIPIcs.MFCS.2022.19.

\bibitem{BollobasKCore}
B. Bollob\'as, The evolution of sparse graphs, in Graph Theory and Combinatorics., Proc. Cambridge Combinatorial Conf. in honor of Paul Erd\H{o}s., Academic Press, 1984, 35--57.

\bibitem{TreesCF}
Y. Caro, M. Petruševski, R. Škrekovski,
Every tree is proper conflict-free (degree+1)-choosable,
Discrete Mathematics 345(1) (2022) 112640.

\bibitem{OuterplanarCF}
Y. Caro, M. Petruševski, R. Škrekovski,
Proper conflict-free degree-choosability of outerplanar graphs,
Applied Mathematics and Computation 449 (2023) 128045.

\bibitem{CPS} Y. Caro, M. Petru\v{s}evski, R. \v{S}krekovski, Remarks on proper conflict-free colorings of graphs, Discrete Math. 346(2) (2023) 113221.

\bibitem{EUN} E. Cho, I. Choi, H. Kwon, B. Park, Proper conflict-free coloring of sparse graphs,	arXiv:2203.16390, 2022.

\bibitem{EUN2} E. Cho, I. Choi, H. Kwon, B. Park, 
Brooks-type theorems for relaxations of square colorings, %	arXiv:2302.06125, 2023.
Discrete Mathematics 348 (2025) 114233.

%\bibitem{hCFBounds}
%E. Cho, I. Choi, H. Kwon, B. Park,
%New bounds for proper $h$-conflict-free colorings,
%arXiv:2309.02458, 2023.

\bibitem{CF-Pirot}
Q. Chuet, T. Dai, Q. Ouyang, F. Pirot,
New bounds for proper $h$-conflict-free colourings,
%arXiv:2505.04543, 2025.
Random Structures Algorithms 68 (2026) e70054.

\bibitem{DWC-CHL}
D.W. Cranston, C.-H. Liu, 	Proper Conflict-free Coloring of Graphs with Large Maximum Degree, 
%	arXiv:2211.02818, 2022.
SIAM Journal on Discrete Mathematics 38 (2024) 10.1137/23M1563281.

\bibitem{deWerra}
D. de Werra, How to color a graph: a survey,
in Combinatorial Programming: Methods and Applications
Proceedings of the NATO Advanced Study Institute held at the Palais des Congr\'es, Versailles, France,  September, 1974, 305--325. 

\bibitem{DebskiPrzybylo} 
M. D\k{e}bski, J. Przyby{\l}o, Conflict-free chromatic number vs conflict-free chromatic index, J. Graph Theory 99(3) (2022) 349--358. https://doi.org/10.1002/jgt.22743.

\bibitem{EvenEtAl} G. Even, Z. Lotker, D. Ron, S. Smorodinsky, Conflict-free colorings of simple geometric regions with applications to frequency assignment in cellular networks, Siam. J. Comput. 33 2003 94--136.

\bibitem{Fab} I. Fabrici, B. Lu\v{z}ar, S. Rindo\v{s}ov\'{a}, R. Sot\'{a}k, Proper conflict-free and unique-maximum colorings of planar graphs with respect to neighborhoods, 
Discrete Appl. Math. 324 (2023) 80--92.

\bibitem{gnp}
A. Frieze and M. Karo\'nski, Introduction to Random Graphs. Cambridge: Cambridge University Press, 2015.

\bibitem{GlebovEtAl} R. Glebov, T. Szab\'o, G. Tardos, Conflict-free coloring of graphs, Combinatorics, Probability and Computing, 23(3) (2014) 434--448.

\bibitem{GuoTrees24}
S. Guo, E.Y.H. Li, L. Li, P. Li, Conflict-free chromatic index of trees, %arXiv:2409.10899, 2024.
Theoretical Computer Science 1058 (2025) 115587.

\bibitem{CFChoosabilityHardness}
S. Gupta, R. Mathew,
Bounds and hardness results for conflict-free choosability,
arXiv:2409.12672,

%\bibitem{Hall}
%P. Hall, On Representatives of Subsets, J. London Math. Soc. 10(1)  (1935) 26--30,
% doi:10.1112/jlms/s1-10.37.26.
%
%\bibitem{RH}
%R. Hickingbotham, Odd colorings, Conflict-Free colorings and Strong coloring Numbers, 
%Australas. J. Combin. 
%87(1) (2023) 160--164.

\bibitem{RandOm} S. Janson, T. {\L}uczak, A. Ruci\'nski, Random Graphs, Wiley, New York, 2000.

\bibitem{MPK}
M. Kamyczura, M. Meszka, J. Przyby{\l}o, A note on the conflict-free chromatic index, Discrete Math. 347(4)  (2024) 113897. 

\bibitem{KPreg}
M. Kamyczura, J. Przyby{\l}o, On asymptotically tight bound for the conflict-free chromatic
index of nearly regular graphs, Discrete Math. 349 (2026) 114945.

\bibitem{KPreg2}
M. Kamyczura, J. Przyby{\l}o, 
On asymptotically tight bounds for the open conflict-free chromatic indexes of nearly regular graphs,
https:2601.17827, 2026.

\bibitem{KP-proper}
M. Kamyczura, J. Przyby{\l}o, On conflict-free proper colorings of graphs without small degree vertices, Discrete Math. 347(1) (2024) 113712.

%\bibitem{KellerEtAl} C. Keller, S. Smorodinsky, Conflict-Free Coloring of Intersection Graphs of Geometric Objects, Discrete Comput Geom, 2019.

%\bibitem{ExtremalCF}
%C. Keller, S. Smorodinsky,
%Extremal results on conflict-free coloring,
%Discrete Geometry and Optimization, Springer (2019) 331–389.

\bibitem{KostochkaEtAl} A. Kostochka, M. Kumbhat and T. {\L}uczak, Conflict-Free colorings of Uniform Hypergraphs With Few Edges, Combinatorics, Probability and Computing 21(4) (2012) 611--622. 

\bibitem{CHL} C.-H. Liu, Proper conflict-free list-coloring, odd minors, 
and layered treewidth, Discrete Math. 347(1) (2024) 113668.

\bibitem{Liu-Reed}
C.-H. Liu and B. Reed. Asymptotically optimal proper conflict-free colouring, Random Structures \& Algorithms
66(3) (2025) e21285, .

\bibitem{Liu-Reed_2}
C.-H. Liu, B. Reed,
Peaceful Colourings,
arXiv:2402.09762, 2024.

\bibitem{LuczakKCore}
T. {\L}uczak, Size and connectivity of the k-core of a random graph, Discrete Mathematics 91(1) (1991) 61--68.

\bibitem{ChernoffBook}
M. Mitzenmacher, E. Upfal, Probability and Computing: Randomized Algorithms and Probabilistic Analysis. Cambridge: Cambridge University Press, 2005.

%\bibitem{MolloyTalagrandRef} M. Molloy and B. Reed, Colouring graphs when the number of colours is almost the maximum degree, J. Combin. Theory Ser. B 109 (2014) 134--195.

\bibitem{PachTardos} J. Pach and G. Tardos, Conflict-free colorings of graphs and hypergraphs, Combinatorics, Probability and Computing 18(5) (2009) 819 -- 834.

\bibitem{PP}
P. P\k{e}ka{\l}a, J. Przyby{\l}o,
On list extensions of the majority edge-colourings of graphs,
The Electronic Journal of Combinatorics 32(4) (2025), Article P4.38,
doi:10.37236/13882.

\bibitem{SmorodinskyPhd} S. Smorodinsky, Combinatorial Problems in Computational Geometry. PhD thesis, School of Computer Science, Tel-Aviv University, 2003.

\bibitem{SmorodinskyApplications} S. Smorodinsky, Conflict-free coloring and its applications. In Geometry--Intuitive, Discrete, and Convex, Springer, 331--389, 2013.

\bibitem{DATA}
Data sharing not applicable to this article as no datasets were generated or analysed during the current study





 \end{thebibliography}
\end{document}